\documentclass[reqno]{amsart}
\usepackage{amsmath,amsthm,amsfonts,amssymb}
\RequirePackage{graphicx}
\numberwithin{equation}{section}
\theoremstyle{plain}

\newtheorem{theorem}{Theorem}

\newtheorem{proposition}[theorem]{Proposition}

\begin{document}

\title[Constrained directed graphs]{Asymptotic structure and singularities 
in constrained directed graphs}

\author{DAVID ARISTOFF}

\author{LINGJIONG ZHU}
\address
{School of Mathematics\newline
\indent University of Minnesota-Twin Cities\newline
\indent 206 Church Street S.E.\newline
\indent Minneapolis, MN-55455\newline
\indent United States of America}
\email{
daristof@umn.edu
\\
zhul@umn.edu}

\date{2 June 2014. \textit{Revised:} 2 June 2014}

  %random graphs, phase transitions, extremal problems 

\begin{abstract}
We study the asymptotics of large directed graphs, 
constrained to have certain densities of edges 
and/or outward $p$-stars. Our models are close 
cousins of exponential random graph models, 
in which edges and certain other subgraph densities are 
controlled by parameters. We find that large graphs have either uniform or bipodal structure. 
When edge density (resp. $p$-star density) is fixed and $p$-star density (resp. edge density) 
is controlled by a parameter, we find phase transitions corresponding to a change from uniform 
to bipodal structure. When both edge and $p$-star density are fixed, we find only bipodal 
structures and no phase transition.
\end{abstract}

\maketitle

\section{Introduction} 

In this article we study the asymptotics of large {directed} graphs 
with constrained densities of edges and outward directed $p$-stars. 
Large graphs are often modeled by probabilistic ensembles with one or 
more adjustable parameters; 
see e.g. Fienberg~\cite{FienbergI,FienbergII}, 
Lov\'{a}sz~\cite{Lovasz2009}
and Newman~\cite{Newman}. 
The {\em exponential random graph models} (ERGMs), in which 
parameters are used to 
tune the densities of edges and other subgraphs, are one such 
class of models; see e.g. Besag \cite{Besag}, Frank and Strauss~\cite{Frank}, 
Holland and Leinhardt \cite{Holland}, Newman~\cite{Newman},
Rinaldo et al.~\cite{Rinaldo}, Robins et al.~\cite{Robins}, 
Snijders et al.~\cite{Snijders}, Strauss \cite{Strauss}, and
Wasserman and Faust~\cite{Wasserman}.

It has been shown that in many 
ERGMs the subgraph densities actually {\em cannot} be tuned. 
For example, for the class of ERGMs parametrized by edges and $p$-stars, 
large graphs are essentially Erd\H{o}s-R\'{e}nyi 
for {all} values of the parameters. (See Chatterjee and Diaconis~\cite{ChatterjeeDiaconis} 
for more complete and precise statements.)
An alternative to ERGMs was introduced by Radin and Sadun~\cite{RadinII}, 
where instead of using parameters to control subgraph counts, 
the subgraph densities are controlled directly; 
see also Radin et al.~\cite{RadinIII}, Radin and Sadun~\cite{RadinIV}
and Kenyon et al.~\cite{Kenyon}. This is the approach we take in this article, 
with edges and (outward) $p$-stars as the subgraphs controlled directly. 
We also consider models which split the difference 
between the two approaches: edges (resp. $p$-stars) 
are controlled directly, while $p$-stars (resp. edges) 
are tuned with a parameter. See~\cite{MeiKenyon} for the 
undirected graph version of the latter.

We find that, in all our models, graphs have either 
uniform or bipodal structure as the number of nodes becomes 
infinite. Our approach, following 
refs.~\cite{Kenyon,RadinII,RadinIII,RadinIV},
is to study maximizers 
of the entropy or free energy of the model as the number 
of nodes becomes infinite. When we constrain both edge 
and $p$-star densities, we find only bipodal structure 
(except when the $p$-star density is fixed to be exactly equal 
to the $p$th power of the edge density). When we constrain 
either edge or $p$-star densities (but not both), we 
find both uniform and bipodal structures, with a sharp 
change at the interface. This is in contrast with the situation in the ERGM 
version of the model, in which one finds 
only uniform structures, albeit with sharp changes in 
their densities along a certain curve in parameter space; 
see Aristoff and Zhu~\cite{AristoffZhu}.

These sharp changes along interfaces 
are called {\em phase transitions}.  
Phase transitions have recently been proved rigorously for ERGMs;  
see e.g. Yin~\cite{Yin} and especially Radin and Yin~\cite{Radin} 
for a precise definition of the term. Some earlier works using mean-field
analysis and other approximations include H\"{a}ggstr\"{o}m and Jonasson~\cite{Haggstrom} and
Park and Newman~\cite{ParkI,ParkII}. 
The terminology is apt, in that ERGMs and our models are inspired 
by certain counterparts in statistical physics: 
respectively, the {\em grand canonical ensemble}, and 
the {\em microcanonical} and {\em canonical ensembles}. 
See the discussion in Section~\ref{MODELS}.

Our directed graph models are simpler than their 
undirected counterparts. In particular, we can 
rigorously identify the asymptotic structures at all 
parameters, while in analogous models for undirected 
graphs, only partial results are known~\cite{Kenyon}. 
Our analysis, however, does not easily extend 
to directed random graph models where other 
subgraphs, like triangles, cycles, etc., instead 
of outward directed $p$-stars, are constrained.

This article is organized as follows. In Section~\ref{MODELS}, 
we describe our models and compare them with their statistical 
physics counterparts. In Section~\ref{RESULTS}, we state our main results. In 
Section~\ref{LDP}, we prove a large deviations 
principle for edge and $p$-star densities. 
We use the large deviations principle to 
give proofs, in Section~\ref{PROOFS}, of our 
main results.

\section{Description of the models}\label{MODELS}

A directed graph on $n$ nodes will be 
represented by a matrix 
$X = (X_{ij})_{1\le i,j\le n}$, where 
$X_{ij} = 1$ if there is a directed 
edge from node $i$ to node $j$, and $X_{ij} = 0$ 
otherwise. For simplicity, we allow for 
$X_{ii} = 1$, though this will not 
affect our results. Let $e(X)$ (resp. $s(X)$) be
the directed edge and outward $p$-star homomorphism 
densities of $X$:
\begin{align}\begin{split}
&e(X):=n^{-2}\sum_{1\le i,j\le n} X_{ij},\\
&s(X):=n^{-p-1} \sum_{1\leq i,j_{1},j_{2},\ldots,j_{p}\leq n}X_{ij_{1}}X_{ij_{2}}\cdots X_{ij_{p}}.
\end{split}
\end{align}
Here, $p$ is an integer $\ge 2$. The reason for the term 
homomorphism density is as follows. 
For a given graph $X$, if $\hbox{hom}_e(X)$ (resp. $\hbox{hom}_s(X)$) 
are the number of homomorphisms -- edge-preserving maps -- 
from a directed edge (resp. outward 
$p$-star) into $X$, then 
\begin{equation*}
e(X) = \frac{\hbox{hom}_e(X)}{n^2},\qquad
s(X) = \frac{\hbox{hom}_s(X)}{n^{p+1}},
\end{equation*}
with the denominators giving the total number of maps 
from a directed edge (resp. outward $p$-star) into $X$.

Let ${\mathbb P}_n$ be the uniform probability measure 
on the set of directed graphs on $n$ nodes. Thus, ${\mathbb P}_n$ 
is the uniform probability measure on the set of $n\times n$ matrices 
with entries in $\{0,1\}$. For $e,s \in [0,1]$ and $\delta > 0$, define
\begin{equation*}
\psi_n^{\delta}(e,s) = \frac{1}{n^2}\log {\mathbb P}_n\left(e(X) \in (e-\delta,e+\delta), \,
s(X) \in (s-\delta,s+\delta)\right).
\end{equation*}
Throughout, $\log$ is the natural logarithm, and we use the convention $0 \log 0 = 0$.
We are interested in the limit
\begin{equation}\label{micropsi}
\psi(e,s) := \lim_{\delta \to 0^+} \lim_{n\to \infty}  \psi_n^{\delta}(e,s).
\end{equation}
The function in~\eqref{micropsi} will be called the {\em limiting 
entropy density}. This is the directed graph version of 
the quantity studied in~\cite{Kenyon}. 
See also~\cite{RadinII,RadinIII,RadinIV} for 
related work where triangles are constrained instead of $p$-stars. 

For $\beta_1,\beta_2 \in {\mathbb R}$, 
$e,s \in [0,1]$ and $\delta > 0$, define
\begin{align*}
&\psi_n^{\delta}(e,\beta_2) = 
\frac{1}{n^2}\log {\mathbb E}_n\left[e^{n^{2}\beta_2 s(X)}\,1_{\{e(X) \in (e-\delta,e+\delta)\}}\right],\\
&\psi_n^{\delta}(\beta_1,s) = 
\frac{1}{n^2}\log {\mathbb E}_n\left[e^{n^{2}\beta_1 e(X)}\,1_{\{s(X) \in (s-\delta,s+\delta)\}}\right].\\
\end{align*}
We will also be interested in the limits
\begin{align}\begin{split}\label{canonpsi}
&\psi(e,\beta_2) := \lim_{\delta \to 0^+} \lim_{n\to \infty} \psi_n^{\delta}(e,\beta_2),\\
&\psi(\beta_1,s) := \lim_{\delta \to 0^+} \lim_{n\to \infty} \psi_n^{\delta}(\beta_1,s).
\end{split}
\end{align}
The quantities in~\eqref{canonpsi} will be called 
{\em limiting free energy densities}. They are 
the directed graph versions of the quantities studied 
in~\cite{MeiKenyon}.
We abuse notation by using the same symbols $\psi$ and 
$\psi_n^{\delta}$ to 
represent different functions in~\eqref{micropsi} and~\eqref{canonpsi} 
(and in~\eqref{ERGM} below), 
but the meaning will be clear from the arguments 
of these functions, which 
will be written
$(e,s)$, $(e,\beta_2)$, $(\beta_1,s)$ (or $(\beta_1,\beta_2)$ as below). When it is clear which function we 
refer to, we may simply write $\psi$ without 
any arguments.

We show that the limits in~\eqref{micropsi} and~\eqref{canonpsi} exist 
by appealing to the variational principles 
in Theorem~\ref{varpsimicro} and 
Theorem~\ref{varpsicanon} below. Maximizers of 
these variational problems are associated with the 
asymptotic structure of the associated constrained 
directed graphs. More precisely, as $n \to \infty$ a typical 
sample $X$ from the uniform measure ${\mathbb P}_n$ 
conditioned on the event $e(X) = e$, $s(X) = s$ 
looks like the maximizer of the variational 
formula for~\eqref{micropsi} in Theorem~\ref{varpsimicro}. Similarly, as $n \to \infty$, 
a typical sample from the probability measure which 
gives weight $\exp[n^2 \beta_2 s]$ (resp. 
$\exp[n^2 \beta_1 e]$) to graphs 
$X$ on $n$ nodes with $s(X) = s$ (resp. 
$e(X) = e$), when conditioned 
on the event $e(X) = e$ (resp. $s(X) = s$),  
looks like the maximizer of the 
variational formula for~\eqref{canonpsi} 
in Theorem~\ref{varpsicanon}. 
Thus, in~\eqref{micropsi} we have 
constrained $e(X)$ and $s(X)$ directly, 
while in~\eqref{canonpsi} we control 
one of $e(X)$ or $s(X)$ with a 
parameter $\beta_1$ or $\beta_2$. 
See~\cite{Kenyon,MeiKenyon,Radin,RadinII, RadinIII, RadinIV} 
for discussions and related work in the undirected graph 
setting. Closely related to~\eqref{micropsi} 
and~\eqref{canonpsi} is the limit
\begin{equation}\label{ERGM}
\psi(\beta_1,\beta_2) := \lim_{n\to \infty} \frac{1}{n^2}
\log {\mathbb E}\left[e^{n^{2}(\beta_1 e(X)+\beta_2 s(X))}\right],
\end{equation}
which was studied 
extensively in Aristoff and Zhu~\cite{AristoffZhu}. 
In this {\em exponential random graph model} (ERGM) both 
$e(X)$ and $s(X)$ are controlled by parameters $\beta_1$ and $\beta_2$. 
There it was shown that $\psi(\beta_1,\beta_2)$  
is analytic except along a certain phase transition 
curve $\beta_2 = q(\beta_1)$, corresponding to 
an interface across which the edge density changes sharply.
See Radin and Yin~\cite{Radin} for similar results, 
which use a similar variational 
characterization, in the 
undirected graph setting. The results 
in~\cite{AristoffZhu} and~\cite{Radin} will 
be useful in our analysis. This is because the Euler-Lagrange equations 
associated with the variational formulas 
for~\eqref{micropsi},~\eqref{canonpsi} 
are the same as the analogous variational 
formula for~\eqref{ERGM}. 
The crucial difference in 
the optimization problems is 
that for the ERGM, 
$\beta_1$ and $\beta_2$ 
are fixed parameters, while in our models 
one or both of $\beta_1,\beta_2$ is a 
free variable (i.e., a Lagrange multiplier). Thus, some solutions 
to the ERGM variational problem may have 
no relevance to 
our models. However, many of the calculations 
can be carried over, which is why we 
can make extensive use of the results 
in~~\cite{AristoffZhu} and~\cite{Radin}.

This article focuses on computing the 
limiting entropy density and free energy densities 
in~\eqref{micropsi} and~\eqref{canonpsi}, along 
with the corresponding maximizers in the variational 
formulas of Theorem~\ref{varpsimicro} and Theorem~\ref{varpsicanon}. 
In statistical physics modeling there is a hierarchy 
analogous to~\eqref{micropsi}-~\eqref{canonpsi}-~\eqref{ERGM}, 
with (for example) particle density and energy density in place of 
$e(X)$ and $s(X)$, and temperature and chemical potential in place 
of $\beta_1$ and $\beta_2$. The statistical physics versions of
~\eqref{micropsi}-~\eqref{canonpsi}-~\eqref{ERGM} correspond 
to the 
{\em microcanonical}, {\em canonical} and {\em grand canonical 
ensembles}, respectively. In that setting, 
there are curves like $q$ along which the free energy 
densities are not analytic, and these correspond to physical 
phase transitions, for example the familiar solid/liquid and liquid 
gas transitions; see Gallavotti~\cite{Gallavotti}. (There is no proof of this statement, 
though it is widely believed; see however Lebowitz et al.~\cite{Lebowitz}.) We find  
singularities in $\psi(e,\beta_2)$ and $\psi(\beta_1,s)$,
but not in $\psi(e,s)$. (In~\cite{AristoffZhu} 
we have shown $\psi(\beta_1,\beta_2)$ has singularities 
as well.) See Radin~\cite{Radin} and
Radin and Sadun~\cite{RadinIV} for discussions 
of the relationship between these 
statistical physics models and 
some random graph models that 
closely parallel ours.

\section{Results}\label{RESULTS}

To state our results we need the following. 
For $x \in [0,1]$, define 
\begin{align*}
&\ell(x) = \beta_1 x + \beta_2 x^p - x \log x - (1-x) \log (1-x),\\
&I(x) = x \log x + (1-x) \log (1-x) + \log 2,
\end{align*}
with the understanding that $I(0) = I(1) = \log 2$.
Of course $\ell$ depends on $\beta_1$ and $\beta_2$, 
but we omit this to simplify notation. Clearly, 
$\ell$ and $I$ are analytic in $(0,1)$ and continuous on $[0,1]$. 
The function $\ell$ is essential to understanding 
the ERGM limiting free energy density~\cite{AristoffZhu,Radin}. 

\begin{theorem}[Radin and Yin~\cite{Radin}]\label{trans_curve}
For each $(\beta_1,\beta_2)$ 
the function $\ell$ has either one or two local maximizers.
There is a curve $\beta_2 = q(\beta_1)$, $\beta_1 \le \beta_1^c$, with the endpoint
\begin{equation*}
(\beta_{1}^{c},\beta_{2}^{c})=\left(\log(p-1) - \frac{p}{p-1},\frac{p^{p-1}}{(p-1)^p}\right),
\end{equation*}
such that off the curve and at the endpoint, $\ell$ has a 
unique global maximizer, while on the curve away from the 
endpoint, $\ell$ has two global 
maximizers $0 < x_1 < x_2 < 1$. We consider $x_1$ and $x_2$ as functions of $\beta_1$ (or $\beta_2$)
for $\beta_1 < \beta_1^c$ (or $\beta_2 > \beta_2^c$); $x_1$ (resp. $x_2$) is increasing (resp. decreasing) 
in $\beta_1$, with
\begin{align}\begin{split}\label{x1_lims}
&\lim_{\beta_1 \to -\infty} x_1 = 0, \quad 
\lim_{\beta_1 \to -\infty} x_2 = 1,\\
&\lim_{\beta_1 \to \beta_1^c} x_1 = \frac{p-1}{p} = 
\lim_{\beta_1 \to \beta_1^c} x_2.\end{split}
\end{align}
\end{theorem}

\begin{theorem}[Aristoff and Zhu~\cite{AristoffZhu}]\label{trans_curve_2}
The curve $q$ in Theorem \ref{trans_curve} is continuous, decreasing, convex, and analytic 
for $\beta_1 < \beta_1^c$, with 
\begin{equation*}
q'(\beta_1) = - \frac{x_2-x_1}{x_2^p-x_1^p}.
\end{equation*}
Moreover, $x_1$ and $x_2$ are analytic in $\beta_1$ and $\beta_2$.
\end{theorem}

The curve in Theorem~\ref{trans_curve} will be called the 
{\em phase transition curve}, and its endpoint the {\em critical point}. 
Theorem~\ref{trans_curve} and Theorem~\ref{trans_curve_2} will be used extensively in most of our proofs; because of this,  
we will often not refer to it explicitly. (We comment that 
the last statement of Theorem~\ref{trans_curve}, not made explicit 
in~\cite{Radin} and~\cite{AristoffZhu}, is 
proved in Proposition~\ref{prop} in Section~\ref{PROOFS} below.) 
We will sometimes write $x_1 = x_2 = (p-1)/p$ for the 
local maximizer of $\ell$ at the critical point. 
Note that the last part of the theorem implies
\begin{equation*}
 0 < x_1 \le \frac{p-1}{p} \le x_2 < 1.
\end{equation*}

Our first main 
result concerns the limiting free energy density 
$\psi(e,s)$.

\begin{theorem}\label{psimicro}
The limiting entropy density $\psi = \psi(e,s)$ is analytic in 
\begin{equation*}
 D:= \{(e,s)\,:\, e^p < s < e,\,0 \le e \le 1\}
\end{equation*}
and equals $-\infty$ outside ${\bar D}$. 
With 
\begin{equation*}
(e^c,s^c) = \left(\frac{p-1}{p},\left(\frac{p-1}{p}\right)^p\right)
\end{equation*}
and 
\begin{equation*}
 D^{up} = \{(e,s)\,:\, s = e,\,0\le e \le 1\},
\end{equation*}
we have the semi-explicit formula 
\begin{equation}\label{psiformula}
\psi(e,s) = \begin{cases} -\frac{x_2-e}{x_2-x_1}I(x_1) 
- \frac{e-x_1}{x_2-x_1}I(x_2), & (e,s) \in D \\
-I(e), & (e,s) \in \partial D \setminus D^{up}\\
-\log 2, & (e,s) \in D^{up}\end{cases},
\end{equation}
where  $0< x_1 < x_2<1$ are 
the unique global maximizers of $\ell$ along the 
phase transition curve $q$ which satisfy 
\begin{equation*}
 \frac{x_2-e}{x_2-x_1}x_1^p + 
 \frac{e-x_1}{x_2-x_1}x_2^p = s.
\end{equation*} 
Moreover, $\psi$ is continuous on ${\bar D}$, 
and the first order partial derivatives of $\psi$ 
are continuous on ${\bar D}\setminus D^{up}$ 
but diverge on $D^{up}$.
When $p=2$, we have the explicit formula
\begin{equation*}
 \psi(e,s) = -I\left(\frac{1}{2}+\sqrt{s - e + \frac{1}{4}}\right), \qquad (e,s) \in {\bar D}.
\end{equation*}
\end{theorem}

In Theorem~\ref{psimicro}, the formula is only semi-explicit 
because $x_1$ and $x_2$ depend on $e$ and $s$. 
Note that the formula depends {\em only} on 
$e$ and $s$. 
(To understand 
the dependence better, see the subsection 
{\em uniqueness of 
the optimizer: geometric proof} in the proof 
of Theorem~\ref{psimicro}.)

Note that we have not found any interesting singular behavior 
of $\psi$, in contrast with the recent results of 
Kenyon et. al.~\cite{Kenyon} 
for the undirected version of the model. On the other hand, 
we are able to prove that our graphs are bipodal, whereas 
in~\cite{Kenyon} it was proved only that the graphs 
are multipodal (though 
simulation evidence from the paper suggests bipodal structure).

\begin{figure}\label{fig1}
\vskip-165pt
\hspace*{-.5cm}\includegraphics[scale=0.66]{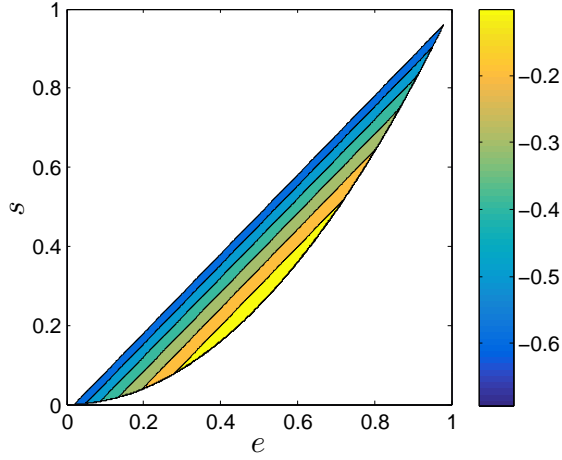}
\vskip-175pt
\caption{Countour plot of $\psi(e,s)$ for $p=2$. 
The non-shaded portion represents the complement 
of ${\bar D}$ (where by convention $\psi$ 
equals $-\infty$). Here, $(e^c,s^c) = (1/2,1/4)$.}
\end{figure}

Our next main result concerns $\psi(e,\beta_2)$ and $\psi(\beta_1,s)$.

\begin{theorem}\label{psicanon}
(i) There is a U-shaped region 
\begin{equation*}
U_e = \{(e,\beta_2)\,:\,x_1 < e < x_2,\,\beta_2 > \beta_2^c\}
\end{equation*}
whose closure has lowest point 
\begin{equation*}
 \left(e^c,\beta_2^c\right) = \left(\frac{p-1}{p}, \frac{p^{p-1}}{(p-1)^p}\right)
\end{equation*}
such that the limiting free energy density $\psi = \psi(e,\beta_2)$ is analytic outside $\partial U_e$. 
The limiting free energy density has the formula
\begin{equation}\label{psiforme}
\psi(e,\beta_{2})=
\begin{cases}
\beta_{2}e^{p}-I(e), &(e,\beta_{2})\in U_e^{c}
\\
\beta_{2}\left[\frac{x_2-e}{x_{2}-x_{1}}x_{1}^{p}
+\frac{e-x_1}{x_{2}-x_{1}}x_{2}^{p}\right]
\\
\qquad\qquad
-\left[\frac{x_2-e}{x_{2}-x_{1}}I(x_{1})
+\frac{e-x_1}{x_{2}-x_{1}}I(x_{2})\right], &(e,\beta_{2})\in U_e
\end{cases}
\end{equation}
where $0 < x_1 < x_2 <1$ are the global maximizers of $\ell$ at 
the point $(q^{-1}(\beta_2), \beta_2)$ on the phase transition curve. 
In particular, $\partial^2\psi/\partial e^2$ has jump discontinuities 
across $\partial U_e$ away from $(e^c, \beta_2^c)$, and 
$\partial^4\psi/\partial e^4$ is discontinuous 
at $(e^c, \beta_2^c)$.

(ii) There is a U-shaped region 
\begin{equation*}
U_s = \{(\beta_1,s)\,:\,x_1^p < s < x_2^p,\,\beta_1 < \beta_1^c\}
\end{equation*}
whose closure has rightmost point 
\begin{equation*}
 \left(\beta_1^c,s^c\right) = \left(\log (p-1) - \frac{p}{p-1}, \left(\frac{p-1}{p}\right)^p\right)
\end{equation*}
such that the limiting free energy density $\psi = \psi(\beta_1,s)$ is analytic 
outside $\partial U_s$. 
The limiting free energy density has the formula
\begin{equation}\label{psiforms}
\psi(\beta_{1},s)=
\begin{cases}
\beta_{1}s^{\frac{1}{p}}-I(s^{\frac{1}{p}}), &(\beta_{1},s)\in U_s^{c}
\\
\beta_{1}\left[\frac{x_{2}^{p}-s}{x_{2}^{p}-x_{1}^{p}}x_{1}
+\frac{s-x_{1}^{p}}{x_{2}^{p}-x_{1}^{p}}x_{2}\right]
\\
\qquad\qquad
-\left[\frac{x_{2}^{p}-s}{x_{2}^{p}-x_{1}^{p}}I(x_{1})
+\frac{s-x_{1}^{p}}{x_{2}^{p}-x_{1}^{p}}I(x_{2})\right], &(\beta_{1},s)\in U_s
\end{cases}
\end{equation}
where $0 < x_1 < x_2 <1$ are the global maximizers of $\ell$ at 
the point $(\beta_1,q(\beta_1))$ on the phase transition curve. 
In particular, $\partial^2\psi/\partial s^2$ has jump discontinuities 
across $\partial U_e$ away from $(\beta_1^c, s^c)$, and 
$\partial^4\psi/\partial s^4$ is discontinuous
at $(\beta_1^c, s^c)$.
\end{theorem}

The sharp change 
in Theorem~\ref{psicanon} is called 
a {\em phase transition} (because of the 
singularity in the derivatives of $\psi$). 
The phase transition corresponds to a 
sharp change from uniform to bipodal 
structure of the optimizers of the 
variational problem for~\eqref{canonpsi};
see the discussion below Theorem~\ref{varpsicanon}. The 
bipodal structure in the regions $U_e$ and $U_s$ is 
sometimes called a {\em replica symmetry breaking phase}. 

In contrast with Theorem~\ref{psimicro}, the $x_1$ and $x_2$ in 
Theorem~\ref{psicanon} are functions of $\beta_1$ or 
$\beta_2$ only and do not depend on $e$ or $s$.
See Figure~1 for the region $D$ from Theorem~\ref{psimicro}, 
and Figure~2 for the U-shaped regions $U_e$ and $U_s$ 
in Theorem~\ref{psicanon}.

\begin{figure}\label{fig2}
\vskip-160pt
\hspace*{-1cm}\includegraphics[scale=0.65]{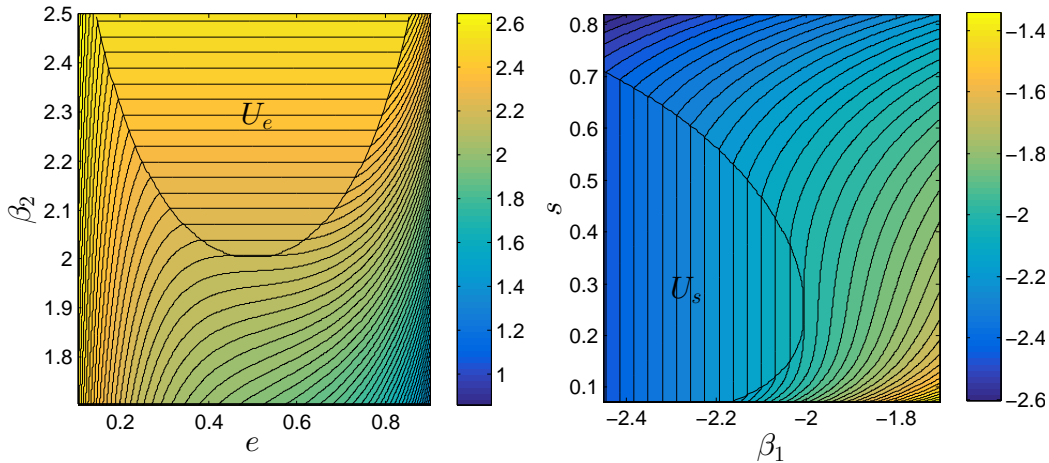}
\vskip-165pt
\caption{Contour plots of $\partial\psi(e,\beta_2)/\partial e$ 
(left) and $\partial\psi(\beta_1,s)/\partial s$ (right) for $p = 2$. 
Contour lines are included to emphasize features. 
The boundaries of the U-shaped regions, $\partial U_e$ and $\partial U_s$, are outlined. 
Here, $(e^c,\beta_2^c) = (1/2, 2)$ 
and $(\beta_1^c,s^c) = (-2,1/4)$.}
\end{figure}

\section{Large Deviations}\label{LDP}

We will need the following terminology before 
proceeding. A sequence 
$({\mathbb Q}_{n})_{n\in\mathbb{N}}$ of 
probability measures on a topological space $X$ 
is said to satisfy a {\em large deviation principle} 
with speed $a_{n}$ 
and rate function 
$J:X\to\mathbb{R}$ if $J$ is non-negative and 
lower semicontinuous, and for any measurable set $A$, 
\begin{equation*}
-\inf_{x\in A^{o}}J(x)\leq\liminf_{n\to\infty}\frac{1}{a_{n}}\log {\mathbb Q}_{n}(A)
\leq\limsup_{n\to\infty}\frac{1}{a_{n}}\log {\mathbb Q}_{n}(A)\leq-\inf_{x\in\bar{A}}J(x).
\end{equation*}
Here, $A^{o}$ is the interior of $A$ and $\bar{A}$ is its closure. 
See e.g. Dembo and Zeitouni~\cite{Dembo} or Varadhan~\cite{VaradhanII}.

We will equip the set $\mathcal G$ of measurable 
functions $[0,1] \to [0,1]$ with  
the {\em cut norm}, written $||\cdot||_\square$ and defined by 
\begin{equation}\label{cutnorm}
 ||g||_\square = \sup \left|\int_A g(x)\,dx\right|,
\end{equation}
where the supremum is taken over measurable subsets $A$ of $[0,1]$.

\begin{theorem}\label{LDPThm}
The sequence of probability measures 
${\mathbb P}_n\left(e(X) \in \cdot, \,
s(X) \in \cdot \right)$ 
satisfies a large deviation principle on the space $[0,1]^2$ with speed $n^{2}$
and rate function
\begin{equation*}
J(e,s)=\inf_{g\in\mathcal{G}_{e,s}}\int_{0}^{1} I(g(x))\,dx,
\end{equation*}
where ${\mathcal G}_{e,s}$ 
is the set of measurable functions 
$g:[0,1] \to [0,1]$ such that 
\begin{equation*}
\int_0^1 g(x)\,dx = e, \quad \int_0^1 g(x)^p\,dx = s.
\end{equation*}
By convention, $I(x) = \infty$ if 
$x \notin [0,1]$, and the infimum is $\infty$ 
if $\mathcal{G}_{e,s}$ is empty.
\end{theorem}

Theorem \ref{psimicro} and Theorem \ref{LDPThm} together imply that
for large number of nodes, a 
typical sample $X$ 
from 
${\mathbb P}_n$ conditioned 
on the event $e(X) \approx e$ and 
$s(X) \approx s$ has the following behavior. 
Approximately $n(x_2-e)/(x_2-x_1)$ of the nodes of $X$
each have on average $nx_{1}$ outward pointing edges, while 
the other approximately $n(e-x_1)(x_2-x_1)$ nodes 
each have on average $nx_{2}$ outward pointing edges. 
When $x_1 \ne x_2$ we call this structure {\em bipodal}; 
otherwise we call it {\em uniform}.

The following is an immediate consequence 
of~Theorem~\ref{LDPThm}.
\begin{theorem}\label{varpsimicro}
For any $e,s \in [0,1]$,
\begin{equation*}
\psi(e,s) = \sup_{{\mathcal G}_{e,s}} \left[-\int_0^1 I(g(x))\,dx\right].
\end{equation*}
\end{theorem}

The next theorem is an easy consequence 
of Theorem~\ref{LDPThm} and Varadhan's lemma 
(see e.g.~\cite{Dembo}).

\begin{theorem}\label{varpsicanon}
For any $e,s \in [0,1]$ and $\beta_1,\beta_2 \in {\mathbb R}$,
\begin{align*}
&\psi(e,\beta_2) = \sup_{{\mathcal G}_{e,\cdot}} \left[\beta_2\int_0^1 g(x)^p\,dx - \int_0^1 I(g(x))\,dx\right]\\
&\psi(\beta_1,s) = \sup_{{\mathcal G}_{\cdot,s}} 
\left[\beta_1 \int_0^1 g(x)\,dx - \int_0^1 I(g(x))\,dx\right],
\end{align*}
where ${\mathcal G}_{e,\cdot}$ (resp. ${\mathcal G}_{\cdot,s}$) 
is the set of measurable functions $g:[0,1]\to [0,1]$ satisfying
\begin{equation*}
\int_0^1 g(x)\,dx = e \quad \left(\hbox{\em resp.} \int_0^1 g(x)^p\,dx = s\right).
\end{equation*}
\end{theorem}

From Theorem~\ref{LDPThm}, the proofs of Theorem~\ref{varpsimicro} 
and Theorem~\ref{varpsicanon} are standard, so we omit them.

Chatterjee and Varadhan~\cite{ChatterjeeVaradhan} 
established large deviations for undirected random 
graphs on the space of graphons (see also Lovasz~\cite{Lovasz}). 
Szemer\'{e}di's lemma was needed
in order to establish the compactness needed for large 
deviations. Since our model consists of directed graphs, the 
results in~\cite{ChatterjeeVaradhan} do not apply directly. 
Theorem~\ref{LDPThm} avoids these technical difficulties:
it is large deviations principle only 
on the space $[0,1]^2$ of edge and star densities, instead 
of on the (quotient) function space of graphons. 
Our proof relies on the simplicity of our edge/directed $p$-star model; 
it cannot be easily extended to the case where edges and 
directed triangles (or other more complicated 
directed subgraphs) are constrained. We expect that these 
models can be handled by adapting the 
results of~\cite{ChatterjeeVaradhan} to the directed case.

In Aristoff and Zhu \cite{AristoffZhu}, it was proved that
\begin{equation}\label{RH}
\lim_{n\to \infty} \frac{1}{n^2}
\log {\mathbb E}\left[e^{n^{2}(\beta_1 e(X)+\beta_2 s(X))}\right]
= \sup_{0\leq x\leq 1}\left(\beta_{1}x+\beta_{2}x^{p}-I(x)\right).
\end{equation}
Observe that the G\"{a}rtner-Ellis theorem cannot 
be used to obtain the large deviations principle
in Theorem \ref{LDPThm} above, due to the fact, 
first observed in Radin and Yin~\cite{Radin}, that 
the right hand side of~\eqref{RH} is not differentiable. 
Instead, we used Mogulskii's theorem and the 
contraction principle
in the proof of Theorem \ref{LDPThm}.

On the other hand, once we have established
the large deviations principle in Theorem \ref{LDPThm}, 
we can use Varadhan's lemma to obtain
an alternative expression for the limiting free energy,
\begin{equation*}
\lim_{n\to \infty} \frac{1}{n^2}
\log {\mathbb E}\left[e^{n^{2}(\beta_1 e(X)+\beta_2 s(X))}\right]
=\sup_{(e,s)\in[0,1]^{2}}\left(\beta_{1}e+\beta_{2}s+\psi(e,s)\right).
\end{equation*}
The limiting free energy in the directed and undirected models
differ by only a constant factor of $1/2$  (see 
Chatterjee and Diaconis~\cite{ChatterjeeDiaconis} and 
Radin and Yin~\cite{Radin}
for the undirected model, and Aristoff and Zhu~\cite{AristoffZhu} 
for the directed model). On the other hand, the limiting 
entropy density $\psi(e,s)$ we 
obtain here differs nontrivially from the one 
recently obtained in Kenyon et al. \cite{Kenyon} 
for undirected graphs.

\section{Proofs}\label{PROOFS}

We start with the proof of the large deviations 
principle of Section~\ref{LDP}. Then we turn to 
the proofs of our main results in Section~\ref{RESULTS}.

\begin{proof}[Proof of Theorem~\ref{LDPThm}]
By definition, $\mathbb{P}_{n}$ is the uniform probability measure on the set of directed graphs on $n$ nodes. Recall that
\begin{align*}
&e(X)=n^{-2}\sum_{1\le i,j\le n} X_{ij},\\
&s(X)=n^{-p-1} \sum_{1\leq i\leq n}\left(\sum_{1\leq j\leq n}X_{ij}\right)^{p}.
\end{align*}
Under $\mathbb{P}_{n}$, $(X_{ij})_{1\leq i,j\leq n}$ are i.i.d. Bernoulli random variables
that take the value $1$ with probability $\frac{1}{2}$ and $0$ with probability $\frac{1}{2}$.
Therefore, the logarithm of the moment generating function of $X_{ij}$ is, for  
$\theta\in\mathbb{R}$,
\begin{equation*}
\log\mathbb{E}[e^{\theta X_{ij}}]=\log\left(\frac{1}{2}e^{\theta}+\frac{1}{2}\right)
=\log(e^{\theta}+1)-\log 2.
\end{equation*}
Its Legendre transform is
\begin{equation*}
\sup_{\theta\in\mathbb{R}}(\theta x-\log\mathbb{E}[e^{\theta X_{ij}}])
=
x\log x+(1-x)\log(1-x)+\log 2=I(x),
\end{equation*}
where by convention $I(x) = \infty$ for $x \notin [0,1]$.
Let $Y_{i}$ be the $i$th entry of the vector
\begin{equation*}
\left(X_{11},X_{12},\ldots,X_{1n},X_{21},X_{22},\ldots,X_{2n},\ldots,X_{n1},X_{n2},\ldots,X_{nn}\right).
\end{equation*}
Then, $Y_{i}$ are i.i.d. Bernoulli random variables.
Mogulskii theorem (see e.g. Dembo and Zeitouni~\cite{Dembo}) shows that 
\begin{equation*} 
\mathbb{P}\left(\frac{1}{n^{2}}\sum_{i=1}^{\lfloor n^{2}x\rfloor}Y_{i}
\in\cdot,0\leq x\leq 1\right)
\end{equation*}
satisfies a sample path large deviation principle
on the space $L_{\infty}[0,1]$ consisting of 
functions on $[0,1]$ equipped with the supremum norm; 
the rate function is given by
\begin{equation*}
\mathcal{I}(G)=
\begin{cases}
\int_{0}^{1}I(G'(x))dx &\text{if $G\in\mathcal{AC}_{0}[0,1]$}
\\
+\infty &\text{otherwise}
\end{cases},
\end{equation*}
where $\mathcal{AC}_{0}[0,1]$ is the set of 
absolutely continuous functions defined on $[0,1]$
such that $G(0)=0$ and $0\leq G'(x)\leq 1$.
The restriction $0\leq G'(x)\leq 1$ comes from the 
fact that $0\leq Y_{i}\leq 1$.
On the other hand, for any $\epsilon>0$,
\begin{align*}
&\limsup_{n\to\infty}\frac{1}{n^{2}}\log\mathbb{P}
\left(\sup_{0\leq x\leq 1}
\left|\frac{1}{n^{2}}\sum_{i=1}^{\lfloor nx\rfloor}
\sum_{j=1}^{n}X_{ij}-\frac{1}{n^{2}}\sum_{i=1}^{\lfloor n^{2}x\rfloor}Y_{i}\right|\geq\epsilon\right)
\\
&\leq\limsup_{n\to\infty}\frac{1}{n^{2}}\log\mathbb{P}
\left(\sup_{0\leq i\leq n^{2}-n}(Y_{i+1}+Y_{i+2}+\cdots Y_{i+n})\geq n^{2}\epsilon\right)
\\
&=-\infty,
\end{align*}
since 
\begin{equation*}
\sup_{0\leq i\leq n^{2}-n}(Y_{i+1}+Y_{i+2}+\cdots Y_{i+n})\leq n.
\end{equation*}
Therefore,
\begin{equation*}
\mathbb{P}\left(\frac{1}{n^{2}}\sum_{i=1}^{\lfloor nx\rfloor}\sum_{j=1}^{n}X_{ij}
\in\cdot,0\leq x\leq 1\right)
\end{equation*} satisfies a large deviation principle with the same
space and rate function as 
\begin{equation*}
\mathbb{P}\left(\frac{1}{n^{2}}\sum_{i=1}^{\lfloor n^{2}x\rfloor}Y_{i}
\in\cdot,0\leq x\leq 1\right).
\end{equation*}

To complete the proof, we need to use the contraction principle, 
see e.g. Dembo and Zeitouni \cite{Dembo} or Varadhan \cite{VaradhanII}.
Given $G\in {\mathcal AC}_0[0,1]$, we may write
$G(x)=\int_{0}^{x}g(y)dy$ for a measurable function $g:[0,1]\to[0,1]$. 
It is easy to see that if $G_n\in {\mathcal AC}_0[0,1]$ 
and $G_{n}\to G$ in the supremum norm, then
$g_{n}\to g$ in the cut norm. 
Hence, if $G_{n}\to G$ in the supremum norm, then 
\begin{equation*}
\int_{0}^{1}g_{n}(x)dx\to\int_{0}^{1}g(x)dx.
\end{equation*}
Moreover, for any $p\geq 2$,
\begin{equation*}
\int_{0}^{1}(g_{n}(x)^{p}-g(x)^{p})dx
=\left[\int_{g_{n}\geq g}(g_{n}(x)^{p}-g(x)^{p})dx\right]
-\left[\int_{g_{n}<g}(g(x)^{p}-g_{n}(x)^{p})dx\right],
\end{equation*}
and since $g_{n}\to g$ in the cut norm and $0\leq g_{n},g\leq 1$, 
the mean value theorem shows that 
\begin{equation*}
\int_{g_{n}\geq g}(g_{n}^{p}(x)-g(x)^{p})dx
\leq p\int_{g_{n}\geq g}(g_{n}(x)-g(x))dx\to 0
\end{equation*}
and
\begin{equation*}
\int_{g_{n}<g}(g^{p}(x)-g_{n}(x)^{p})dx
\leq p\int_{g_{n}<g}(g(x)-g_{n}(x))dx\to 0.
\end{equation*}
Therefore, the maps
\begin{equation*}
G\mapsto\int_{0}^{1}G'(x)dx,\qquad
G\mapsto\int_{0}^{1}(G'(x))^{p}dx,
\end{equation*}
are continuous from $L_{\infty}[0,1]\cap\mathcal{AC}_{0}[0,1]$ to $[0,1]$ 
and thus the map
\begin{equation*}
G\mapsto\left(\int_{0}^{1}G'(x)dx,\int_{0}^{1}(G'(x))^{p}dx\right)
\end{equation*}
is continuous from $L_{\infty}[0,1]\cap\mathcal{AC}_{0}[0,1]$ to $[0,1]^{2}$. 
By the contraction principle, we conclude that ${\mathbb P}_n\left(e(X) \in \cdot, \,
s(X) \in \cdot \right)$ satisfies a large deviation principle on the space $[0,1]^2$ with speed $n^{2}$
and rate function $J(e,s)=\inf_{g\in\mathcal{G}_{e,s}}\int_{0}^{1} I(g(x))\,dx.$
\end{proof}

The following proofs are for our main results in 
Section~\ref{RESULTS}.

\begin{proof}[Proof of Theorem~\ref{psimicro}]

By Theorem~\ref{varpsimicro}, 
\begin{equation}\label{functional}
\psi(e,s) = \sup_{g \in {\mathcal G}_{e,s}} 
\left[-\int_0^1 I(g(x))\,dx\right]
\end{equation}
where we recall ${\mathcal G}_{e,s}$ 
is the set of measurable functions $g:[0,1]\to [0,1]$ 
satisfying
\begin{equation}\label{C1}
\int_0^1 g(x)\,dx = e, \quad \int_0^1 g(x)^p \,dx = s.
\end{equation}

{\em Domain of $\psi(e,s)$.} By Jensen's inequality, 
\begin{equation*}
s = \int_0^1 g(x)^p \,dx 
\ge \left(\int_0^1 g(x)\,dx\right)^p = e^p.
\end{equation*}
On the other hand, since $g$ takes values in 
$[0,1]$ and $p > 1$, 
\begin{equation*}
s = \int_0^1 g(x)^p \,dx \le \int_0^1 g(x)\,dx = e.
\end{equation*}
Hence by convention $\psi(e,s)$ is infinite outside 
${\bar D}$. We will show later in the proof that $\psi(e,s)$ 
is finite in ${\bar D}$. 

{\em Optimizing over a quotient space.}
So that we can use standard variational techniques, we 
show that it suffices to optimize over a compact 
quotient space of ${\mathcal G}_{e,s}$.  
Recall that ${\mathcal G}$ is the set of measurable 
functions $[0,1] \to [0,1]$ endowed with the cut 
norm~\eqref{cutnorm}. Let ${\tilde {\mathcal G}}$ be 
the quotient space obtained by the following equivalence 
relation: $g \sim h$ if and only if 
there is a measure-preserving bijection $\sigma:[0,1] \to [0,1]$ 
such that $g = h\circ \sigma$. Note that 
for any such $\sigma$ we have $||g||_\square = ||g \circ \sigma||_\square$, and so 
\begin{equation*}
 \delta_\square({\tilde g},{\tilde h}) := \inf_\sigma ||g - h\circ \sigma||_\square
\end{equation*}
defines a metric on ${\tilde {\mathcal G}}$. With 
this metric ${\tilde {\mathcal G}}$ is compact~\cite{Janson}. 
Observe that 
\begin{equation*}
 -\int_0^1 I(g(x))\,dx,\quad \int_0^1 g(x)\,dx, \quad \int_0^1 g(x)^p\,dx
\end{equation*}
are all unchanged when $g$ is replaced by $g \circ \sigma$. Thus, 
\begin{equation*}
\psi(e,s) = \sup_{g \in {\mathcal G}_{e,s}} 
\left[-\int_0^1 I(g(x))\,dx\right] = \sup_{{\tilde g} \in {\tilde {\mathcal G}}_{e,s}} 
\left[-\int_0^1 I(g(x))\,dx\right],
\end{equation*}
where ${\tilde {\mathcal G}}_{e,s}$ is the projection of 
${\mathcal G}_{e,s}$ in the quotient space, and on the right 
hand side $g$ is any element of the equivalence class of ${\tilde g}$. 
Next we check 
that the functionals  
\begin{equation}\label{cts}
 {\tilde g}\mapsto -\int_0^1 I(g(x))\,dx,\quad   {\tilde g}\mapsto \int_0^1 g(x)\,dx,
 \quad {\tilde g}\mapsto \int_0^1 g(x)^p\,dx,
\end{equation}
defined on $\tilde {\mathcal G}$, 
are all continuous. Let ${\tilde g}_n$ be a sequence in ${\tilde {\mathcal G}}$
converging to ${\tilde g}$, and let 
$g_n, g$ be representatives for 
${\tilde g}_n, {\tilde g}$, respectively. We will show that
\begin{equation}\label{toshow}
\left|\int_0^1 F(g(x))\,dx - \int_0^1 F(g_n(x))\,dx\right| \to 0 \hbox{ as } n \to \infty
\end{equation}
for any uniformly continuous function $F$ defined on $[0,1]$. 
For each $n$, choose $\sigma_n$ such that 
\begin{equation*}
 ||g - g_n\circ \sigma_n||_\square \to 0 \hbox{ as } n\to \infty.
\end{equation*}
For any $\delta > 0$, define
\begin{equation*}
 A_{n,\delta} = \{x\,:\,|g(x) - g_n(\sigma_n(x))| \ge \delta\}.
\end{equation*}
As the cut norm is equivalent to the $L^1$ norm (this is true only 
in one dimension; see~\cite{Janson}), 
we have for any $\delta > 0$, 
\begin{equation}\label{An}
 |A_{n,\delta}| \to 0 \hbox{ as } n \to \infty,
\end{equation}
where $|A_{n,\delta}|$ is the Lebesgue measure of $A_{n,\delta}$.
Now, observe that 
\begin{align}\begin{split}\label{EST}
&\left|\int_0^1 F(g(x))\,dx - \int_0^1 F(g_n(x))\,dx\right| \\
&= \left|\int_0^1 F(g(x))\,dx - \int_0^1 F(g_n\circ \sigma_n(x))\,dx\right| \\
&\le 2|A_{n,\delta}|\sup_{x \in [0,1]}|F(x)| + \left|\int_{[0,1]\setminus A_{n,\delta}} [F(g(x)) -F(g_n\circ \sigma_n(x))]\,dx \right|.
\end{split}
\end{align}
Let $\epsilon > 0$. Using uniform continuity of $F$, let $\delta>0$ be 
such that $|x-y|<\delta$ implies $|F(x)-F(y)| < \epsilon$. Then~\eqref{EST} 
shows that 
\begin{equation}\label{lastf}
 \left|\int_0^1 F(g(x))\,dx - \int_0^1 F(g_n(x))\,dx\right| 
\le 2|A_{n,\delta}|\sup_{x \in [0,1]}|F(x)| + \epsilon.
\end{equation}
Now~\eqref{An} establishes~\eqref{toshow}, as desired. 
These arguments show that there exist (global) maximizers $g$ 
of 
\begin{equation}\label{functionalcomp}
\psi(e,s) = \sup_{{g} \in{{\mathcal G}}_{e,s}}\left[-\int_0^1 I(g(x))\,dx \right].
\end{equation}

{\em Bipodal structure of the optimizers.}
If $(e,s) \in D^{up}$, then for any ${g} \in {{\mathcal G}}_{e,s}$,
\begin{equation*}
\left|\left\{x\,:\,{g}(x)\in \{0,1\}\right\}\right| = 1.
\end{equation*} 
Suppose then that $(e,s) \in {\bar D}\setminus D^{up}$. 
Since $I'(x) \to -\infty$ as $x \to 0$ and 
$I'(x) \to \infty$ as $x \to 1$, it is not hard to see that, 
given $(s,e) \in {\bar D}\setminus D^{up}$, 
there exists $\epsilon > 0$ such that any optimizers of~\eqref{functionalcomp}
must be in the following set:
\begin{equation}\label{gepsilon}
 {\mathcal G}^{\epsilon} := \{g\in {\mathcal G}\,:\,g(x) \in [\epsilon,1-\epsilon] \hbox{ for a.e. } 
x \in [0,1]\}.
\end{equation}
To see this, note that if $g \in {\mathcal G}_{e,s}\setminus {\mathcal G}^{\epsilon}$, 
then the values of $g$ can be adjusted to be $\epsilon$ distance away from $0$ and 
$1$ so that still $g \in {\mathcal G}_{e,s}$ and $g$ attains a 
larger value for the integral in~\eqref{functionalcomp}.
Then, optimizing over ${\mathcal G}^{\epsilon}$, standard results in variational calculus 
(see~\cite{Clarke}, Theorem~9.1, pg.~178) show that optimizers $g$ 
satisfy, for all $\delta g \in {\mathcal G}$, 
\begin{equation}\label{VC}
-\eta\int_0^1 I'(g(x))\delta g(x)\,dx + 
\beta_1\int_0^1 \delta g(x)\,dx + 
\beta_2 \int_0^1 pg(x)^{p-1}\delta g(x)\,dx = 0,
\end{equation}
where $\eta \in \{0,1\}$, $\beta_1,\beta_2 \in {\mathbb R}$ 
are Lagrange multipliers, and at least one 
of $\eta$, $\beta_1$, $\beta_2$ is nonzero. The 
integrals in~\eqref{VC} are the Fr\'echet derivatives of 
\begin{equation*}
 {g}\mapsto \int_0^1 I(g(x))\,dx,\quad   {g}\mapsto \int_0^1 g(x)\,dx,
 \quad {g}\mapsto \int_0^1 g(x)^p\,dx,
\end{equation*}
evaluated at $\delta g \in {\mathcal G}$; a simplified version of 
the arguments in~\eqref{toshow}-~\eqref{lastf} shows the Fr\'echet 
derivatives are continuous in ${g} \in {\mathcal G}^{\epsilon}$. 
(To apply the results mentioned in~\cite{Clarke}, 
we must switch to the uniform topology induced by the sup norm 
$||\cdot||_\infty$ on the 
Banach space of bounded measurable functions $[0,1] \to {\mathbb R}$. 
The statement about 
continuity of Frechet derivatives then follows from the 
simple fact that $||g-h||_\infty < \epsilon$ 
implies $||g-h||_1 < \epsilon$.)

When $\eta = 1$ (called the 
{\em normal} case), we see that for some $\beta_1$, $\beta_2$, 
\begin{equation*}
\ell'(g(x)) = 0,\quad \hbox{ for a.e. } x \in [0,1].
\end{equation*}
When $\eta = 0$ (called the {\em abnormal} case), we 
find that $\beta_1 + p\beta_2 g(x)^{p-1} = 0$ for a.e.
$x \in [0,1]$, so that $g$ is constant a.e. 
From~\eqref{C1} and Jensen's inequality, this 
occurs precisely when $s = e^p$. 

Let ${g}$ be a maximizer of~\eqref{functional}. 
We have shown that: 
\begin{align*}
&(a) \hbox{ if } (e,s) \in D^{up},  
\hbox{ then } g(x) \in \{0,1\}  \hbox{ for a.e. }x \in [0,1];\\ 
&(b) \hbox{ if } (e,s) \in D, \hbox{ then } 
\ell'(g(x)) = 0 
\hbox{ for a.e. }x \in [0,1]; \hbox { or}\\
&(c) \hbox{ if }
s \in {\partial D}\setminus D^{up}, 
\hbox{ then }g \hbox{ is constant a.e.}
\end{align*}
Theorem~\ref{trans_curve} shows 
that, for each $(\beta_1,\beta_2)$, 
either $\ell'(y) = 0$ 
at a unique $y \in (0,1)$ or $\ell'(y) = 0$ at 
exactly two points $0< y_1 < y_2 < 1$. 
Thus, to maximize~\eqref{functional} 
it suffices to maximize
\begin{equation*}
-\int_0^1 I(g(x))\,dx
\end{equation*}
over the set of functions $g \in {\mathcal G}_{e,s}$ of the form 
\begin{equation}\label{form}
g(x) = \begin{cases} y_1, & x \in A\\
y_2, & x \notin A\end{cases}
\end{equation}
where $|A| = \lambda \in [0,1]$ and $0 \le y_1 \le y_2 \le 1$. 
Observe that for such $g$, 
\begin{align*}
-\int_0^1 I(g(x))\,dx &= -\lambda I(y_1) - (1-\lambda)I(y_2),\\
\int_0^1 g(x)\,dx &= \lambda y_1 + (1-\lambda)y_2,\\
\int_0^1 g(x)^p\,dx &= \lambda y_1^p + (1-\lambda)y_2^p.
\end{align*}
It is therefore enough to maximize
\begin{equation}\label{MAX}
-\lambda I(y_1) - (1-\lambda)I(y_2)
\end{equation}
subject to the constraints 
\begin{align}\begin{split}\label{constraints2}
&\lambda y_1 + (1-\lambda)y_2 = e,\\
&\lambda y_1^p + (1-\lambda)y_2^p = s.\end{split}
\end{align}
In case {\em (a)}, we must have $y_1 = 0$, 
$y_2 = 1$ and $1-\lambda = e = s$. 
In case {\em (c)}, we must take $y_1 = y_2 = e$. 
This establishes formula~\eqref{psiformula} for $(e,s) \in \partial D$. 
It remains to consider {\em (b)}. In this 
case $(e,s) \in D$ 
we have seen that $0 < y_1 \le y_2 < 1$. 
Moreover if $y_1 = y_2$ or 
$\lambda \in \{0,1\}$ then~\eqref{constraints2} 
cannot be satisfied, so $y_1 < y_2$ and 
$\lambda \in (0,1)$. 
Introducing Lagrange multipliers $\beta_1$, 
$\beta_2$, we see that 
\begin{align*}
&-\lambda I'(y_1) + \lambda\beta_1  + \lambda p\beta_2 y_1^{p-1} = 0,\\
&-(1-\lambda) I'(y_2) + (1-\lambda)\beta_1  + (1-\lambda) p\beta_2 y_2^{p-1} = 0,\\
&-[I(y_1) - I(y_2)] + \beta_1(y_1-y_2) + \beta_2(y_1^p-y_2^p) = 0.
\end{align*}
Thus,
\begin{equation*}
\ell'(y_1) = \ell'(y_2) = 0, \quad \ell(y_1) = \ell(y_2).
\end{equation*}
This means that $y_1 = x_1 < x_2 = y_2$ 
are global maximizers of $\ell$ on the phase 
transition curve, away from the critical point.

{\em Uniqueness of the optimizer (geometric proof).} 
We must prove uniqueness for $(e,s) \in {D}$. 
Consider the class of lines 
\begin{equation}\label{lines}
\lambda(x_1,x_1^p) + (1-\lambda)(x_2,x_2^p), \quad 0<\lambda < 1,
\end{equation}
where $0<x_1 < x_2 < 1$ are 
global maximizers of $\ell$ 
on the phase transition curve, away from the critical point. 
Since $x_1$ (resp. $x_2$) is strictly increasing (resp. strictly 
decreasing) in $\beta_1$ with $x_1 \le x_2$, 
no two distinct lines of this class 
can intersect. Continuity of $x_1$, $x_2$ in $\beta_1$ 
along with~\eqref{x1_lims} show that the union of all the lines  
equals ${D}$. Thus, given $(e,s) \in {D}$, there is a 
unique optimizer $g$ of the form~\eqref{form}, obtained by 
locating the unique line from~\eqref{lines} which contains 
the point $(e,s)$ and choosing $\lambda$ satisfying 
the constraint~\eqref{constraints2}. Solving 
for $\lambda$ in the first constraint, we 
get $\lambda = (x_2-e)/(x_2-x_1)$. 
Putting this into the second constraint gives
\begin{equation*}
\frac{x_2-e}{x_2-x_1}x_1^p + \frac{e-x_1}{x_2-x_1}x_2^p = s.
\end{equation*}
This establishes formula~\eqref{psiformula} for $(e,s) \in D$.

{\em Uniqueness of the optimizer (algebraic proof).} 
Fix $(e,s) \in D$.  
From~\eqref{constraints2}, 
\begin{equation*}
\lambda=\frac{e-x_{2}}{x_{2}-x_{1}}
=\frac{s-x_{2}^{p}}{x_{2}^{p}-x_{1}^{p}}
\end{equation*}
and so
\begin{equation}\label{sol}
e-x_{2} = -q'(\beta_{1})(s-x_{2}^{p}).
\end{equation}
We must show that~\eqref{sol} has a unique solution. 
Define
\begin{equation*}
F(\beta_{1}):=e-x_{2}+q'(\beta_{1})(s-x_{2}^{p}).
\end{equation*}
Note that  
\begin{equation*}
\lim_{\beta_{1}\to-\infty}F(\beta_{1})
=e-1-(s-1)=e-s\geq 0,
\end{equation*}
and
\begin{align*}
\lim_{\beta_{1}\to\beta_{1}^{c}}F(\beta_{1})
&=e-\frac{p-1}{p}-\frac{p^{p-2}}{(p-1)^{p-1}}\left(s-\left(\frac{p-1}{p}\right)^{p}\right)
\\
&=e-\left(\frac{p-1}{p}\right)^{2}-s\frac{p^{p-2}}{(p-1)^{p-1}}
\\
&\leq e-\left(\frac{p-1}{p}\right)^{2}-e^{p}\frac{p^{p-2}}{(p-1)^{p-1}}.
\end{align*}
Also define 
\begin{equation*}
G(x):=x-\left(\frac{p-1}{p}\right)^{2}-x^{p}\frac{p^{p-2}}{(p-1)^{p-1}},\quad 0\leq x\leq 1.
\end{equation*}
Then $G(0)<0$ and $G(1)<0$. Moreover,
\begin{equation*}
G'(x)=1-x^{p-1}\left(\frac{p}{p-1}\right)^{p-1}
\end{equation*}
is positive if $0<x<\frac{p-1}{p}$ and it is negative
if $\frac{p-1}{p}<x<1$. Finally, 
\begin{equation*}
G\left(\frac{p-1}{p}\right)=\frac{p-1}{p}-\frac{(p-1)^{2}}{p^{2}}-\frac{p-1}{p^{2}}=0.
\end{equation*}
Therefore, 
\begin{equation*}
\lim_{\beta_{1}\to\beta_{1}^{c}}F(\beta_{1})
\leq e-\left(\frac{p-1}{p}\right)^{2}-e^{p}\frac{p^{p-2}}{(p-1)^{p-1}}\leq 0.
\end{equation*}
Now by the Intermediate Value Theorem, there exists 
$\beta_{1}^{\ast}\leq\beta_{1}^{c}$ such that 
$F(\beta_{1}^{\ast})=0$.

Next we prove $\beta_{1}^{\ast}$ is unique. Observe that
\begin{equation*}
F'(\beta_1)
=-\left[1+q'(\beta_{1})px_{2}^{p-1}\right]\frac{\partial x_{2}}{\partial\beta_{1}}
+q''(\beta_{1})(s-x_{2}^{p}).
\end{equation*}
Since $0\leq\lambda\leq 1$, we have $s-x_{2}^{p}\leq 0$. We also know that $q''(\beta_{1})\geq 0$
and $\partial x_{2}/\partial\beta_{1}>0$.
Moreover, since 
\begin{equation*}
q'(\beta_{1})=-\frac{x_{1}-x_{2}}{x_{1}^{p}-x_{2}^{p}}<0
\end{equation*}
and $0<x_{1}<\frac{p-1}{p}<x_{2}<1$, we have
\begin{equation*}
1+q'(\beta_{1})px_{2}^{p-1}=1-\frac{x_{1}-x_{2}}{x_{1}^{p}-x_{2}^{p}}px_{2}^{p-1}<0.
\end{equation*} 
Hence, we conclude that $\partial F/\partial\beta_{1}<0$ for $\beta_{1}<\beta_{1}^{c}$
and therefore $\beta_{1}^{\ast}$ is unique.

{\em Regularity of $\psi(e,s)$.}
We turn now to the claimed regularity of $\psi(e,s)$, 
starting with analyticity. 

Fix $e \in (0,1)$. Then each $s \in (e^p,e)$ satisfies, for 
some $x_1 < x_2$, 
\begin{equation}\label{star2}
 \frac{x_2-e}{x_2-x_1}x_1^p + 
 \frac{e-x_1}{x_2-x_1}x_2^p = s.
\end{equation}  
We claim that~\eqref{star2} defines $\beta_1$ 
implicitly as an analytic function of $s$ 
for $s \in (e^p,e)$. By differentiating 
the left hand side of~\eqref{star2} with respect to $\beta_1$, 
we find the expression
\begin{align*}
&\left(\frac{x_1-e}{x_2-x_1}\frac{\partial x_2}{\partial \beta_1}
+ \frac{e-x_2}{x_2-x_1}\frac{\partial x_1}{\partial \beta_1}\right)
\frac{x_2^p-x_1^p}{x_2-x_1}\\
&\quad \quad - \frac{x_1-e}{x_2-x_1}\frac{\partial x_2}{\partial \beta_1}px_2^{p-1} 
- \frac{e-x_2}{x_2-x_1}\frac{\partial x_1}{\partial \beta_1}px_1^{p-1}.
\end{align*}
By the mean value theorem, there is $x_1 < y < x_2$ such that 
this expression becomes
\begin{equation*}
\left[\frac{x_1-e}{x_2-x_1}\frac{\partial x_2}{\partial \beta_1}
(py^{p-1}-px_2^{p-1})\right] + 
\left[\frac{e-x_2}{x_2-x_1}\frac{\partial x_1}{\partial \beta_1}
(py^{p-1}-px_1^{p-1})\right].
\end{equation*}
Since $\partial x_1/\partial \beta_1 > 0$ 
and $\partial x_2/\partial \beta_1 < 0$,
each of the terms in brackets is negative, so this 
expression is nonzero. By the analytic implicit function 
theorem~\cite{Krantz}, we conclude that $\beta_1$ 
is an analytic function of 
$s$ inside $D$. Similar arguments show that $\beta_1$ 
is an analytic function of $e$ inside $D$. (By 
Theorem~1, this means $\beta_2$ must also be an analytic 
function of $e$ and $s$ inside $D$.)
Since $x_1$ and $x_2$ are analytic functions 
of $\beta_1$, we see that $x_1$ and $x_2$ are analytic 
functions of $e$ and $s$ inside $D$. Inspecting~\eqref{psiformula}, 
we conclude that $\psi = \psi(e,s)$ is analytic in ${D}$.

Next we show that $\psi$ is 
continuous on ${\bar D}$. We have already shown that 
$\psi$ is analytic, hence continuous, in ${D}$. It 
is easy to check that $\psi$ is 
continuous along each line 
\begin{equation*}
\lambda(x_1,x_1^p) + (1-\lambda)(x_2,x_2^p), \quad 0\le\lambda \le 1,
\end{equation*} 
and that the restriction 
of $\psi$ to the lower boundary 
$\{(e,s)\,:\, s = e,\,0< e< 1\}$ is continuous. 
This is enough to conclude that $\psi$ 
is continuous on ${\bar D}\setminus D^{up}$.
Finally for $(e_0,e_0) \in D^{up}$, we have 
\begin{equation*}
 \lim_{(e,s) \to (e_0,e_0)} \psi(e,s) = 
\lim_{x_1 \to 0, \, x_2 \to 1} -\frac{x_2-e}{x_2-x_1}I(x_1) 
- \frac{e-x_1}{x_2-x_1}I(x_2) = -\log 2 = \psi(e_0,e_0).
\end{equation*}

Next we prove continuity of the first order 
derivatives on ${\bar D}\setminus D^{up}$. It suffices to show that: 
(i) the first order partial derivatives of 
$\psi$ are continuous 
in $D$; and (ii) the limits of the first order 
partial derivatives of $\psi$ at the lower 
boundary $\{(e,s)\,:\,s = e^p, \, 0 < e < 1\}$ exist. 
With 
\begin{equation*}
 \lambda = \frac{x_2-e}{x_2-x_1}
\end{equation*}
and using~\eqref{psiformula}, we see that for $(e,s) \in D$, 
\begin{equation}\label{1stderiv}
 \frac{\partial \psi}{\partial s} = \frac{\partial \lambda}{\partial s}\left(I(x_1)-I(x_2)\right) 
+ \lambda I'(x_1) \frac{\partial x_1}{\partial s}
+ (1-\lambda)I'(x_2)\frac{\partial x_2}{\partial s}.
\end{equation}
Using the fact that $I(x) = \beta_1 x + \beta_2 x^p - \ell(x)+\log 2$,
\begin{align*}
 \frac{\partial \psi}{\partial s} &= 
\frac{\partial \lambda}{\partial s}\left(\beta_1 x_1 + \beta_2 x_1^p 
- \beta_1 x_2 - \beta_2 x_2^p\right) \\
&\quad \quad+ \lambda \frac{\partial x_1}{\partial s}\left(\beta_1  + p\beta_2 x_1^{p-1}\right) 
+ (1-\lambda)\frac{\partial x_2}{\partial s}\left(\beta_1 + p\beta_2 x_2^{p-1}\right).
\end{align*}
It is straightforward to compute that 
\begin{equation}\label{plambda}
 \frac{\partial \lambda}{\partial s} = \lambda\frac{\partial x_1}{\partial s}\frac{1}{x_2-x_1}
+ (1-\lambda)\frac{\partial x_2}{\partial s}\frac{1}{x_2-x_1}.
\end{equation}
Thus, 
\begin{align*}
  \frac{\partial \psi}{\partial s} &= 
- \left(\lambda\frac{\partial x_1}{\partial s}+(1-\lambda)\frac{\partial x_2}{\partial s}\right)
\left(\beta_1 + \beta_2 \frac{x_2^p-x_1^p}{x_2-x_1}\right)\\
&\quad \quad+ \lambda \frac{\partial x_1}{\partial s}\left(\beta_1  + p\beta_2 x_1^{p-1}\right) 
+ (1-\lambda)\frac{\partial x_2}{\partial s}\left(\beta_1  + p\beta_2 x_2^{p-1}\right)\\
&= \beta_2\left[\lambda\frac{\partial x_1}{\partial s}\left(p x_1^{p-1}- \frac{x_2^{p}-x_1^p}{x_2-x_1}\right)
+(1-\lambda)\frac{\partial x_2}{\partial s}\left(p x_2^{p-1}- \frac{x_2^{p}-x_1^p}{x_2-x_1}\right)\right].
\end{align*}
Differentiating 
\begin{equation*}
 \lambda x_1^p + (1-\lambda)x_2^p = s
\end{equation*}
with respect to $s$ and using~\eqref{plambda}, we find that 
\begin{equation*}
 1 = \lambda\frac{\partial x_1}{\partial s}\left(p x_1^{p-1}- \frac{x_2^{p}-x_1^p}{x_2-x_1}\right)
+(1-\lambda)\frac{\partial x_2}{\partial s}\left(p x_2^{p-1}- \frac{x_2^{p}-x_1^p}{x_2-x_1}\right),
\end{equation*}
and so
\begin{equation*}
  \frac{\partial \psi}{\partial s} = 
\beta_2.
\end{equation*}
We have already seen that $\beta_2$ is an analytic, hence continuous, 
function of $(e,s)$ inside $D$. Since 
\begin{equation*}
q^{-1}(\beta_2) + p\beta_2 x_i^{p-1} - I'(x_i) = 0, \quad i=1,2,
\end{equation*}
we see that for $e_0 \in (0,1)$, 
\begin{align*}
\lim_{(e,s) \to (e_0,e_0^p)} \beta_2 
&= \lim_{x_i \to e_0} \frac{I'(x_i) - q^{-1}(\beta_2)}{px_i^{p-1}} \\
&= \frac{I'(e_0)-\beta_1^c}{pe_0^{p-1}}
\end{align*}
where $i=1$ if $e_0 \in (0,(p-1)/p)$ and otherwise $i=2$. 
Also, for any $e_0 \in [0,1]$,
\begin{equation*}
 \lim_{(e,s) \to (e_0,e_0)} \beta_2 = \lim_{x_1 \to 0} \beta_2 = \infty.
\end{equation*}
Proofs of regularity for $\partial \psi/\partial e$ are 
similar so we omit them.

{\em Explicit formula when $p = 2$.} 
When $p=2$, we can explicitly solve
\begin{align*}
&\lambda x_1 + (1-\lambda)x_2 = e\\
&\lambda x_1^2 + (1-\lambda)x_2^2 = s
\end{align*}
to obtain 
\begin{equation*}
x_{1}=\frac{1-\sqrt{1-4(e-s)}}{2}
\qquad
x_{2}=\frac{1+\sqrt{1-4(e-s)}}{2}
\end{equation*}
and 
\begin{equation*}
\lambda=\frac{x_2-e}{x_{2}-x_{1}}
=\frac{\sqrt{1-4(e-s)}+1-2e}{2\sqrt{1-4(e-s)}}.
\end{equation*}
This yields
\begin{align*}
\psi(e,s)&=-\lambda I(x_{1})-(1-\lambda)I(x_{2})
\\
&=-\left(\frac{1}{2}+\frac{1-2e}{2\sqrt{1-4(e-s)}}\right)
I\left(\frac{1}{2}-\frac{1}{2}\sqrt{1-4(e-s)}\right)
\\
&\qquad
-\left(\frac{1}{2}-\frac{1-2e}{2\sqrt{1-4(e-s)}}\right)
I\left(\frac{1}{2}+\frac{1}{2}\sqrt{1-4(e-s)}\right)
\\
&=
-I\left(\frac{1}{2}+\frac{1}{2}\sqrt{1-4(e-s)}\right),
\end{align*}
where the last line uses the fact that $I$ is symmetric around $1/2$.
\end{proof}

\begin{proof}[Proof of theorem~\ref{psicanon}]
{\em Proof of part (i).} 
By Theorem~\ref{varpsicanon}, 
\begin{equation}\label{functional2}
\psi(e,\beta_2) = \sup_{{\mathcal G}_{e,\cdot}} \left[\beta_2\int_0^1 g(x)^p\,dx - \int_0^1 I(g(x))\,dx\right],
\end{equation}
where we recall ${\mathcal G}_{e,\cdot}$ 
is the set of measurable functions $g:[0,1]\to [0,1]$ satisfying
\begin{equation}\label{constraint_e}
\int_0^1 g(x)\,dx = e.
\end{equation}
Arguments similar to those in the 
proof of Theorem~\ref{psimicro} show that 
optimizers of~\eqref{functional2} exist 
and have the form 
\begin{equation}\label{form2}
g(x) = \begin{cases} y_1, & x \in A \\ y_2, & x \notin A\end{cases}
\end{equation}
where $|A| = \lambda \in [0,1]$ and $0 \le y_1 \le y_2 \le 1$. 
For such functions $g$, we have 
\begin{align*}
\beta_2\int_0^1 g(x)^p\,dx - \int_0^1 I(g(x))\,dx 
&=  \beta_2\left[\lambda y_1^p + 
(1-\lambda) y_2^p\right]- \left[\lambda I(y_1) + (1-\lambda) I(y_2)\right],\\
 \int_0^1 g(x)\,dx &= \lambda y_1 + (1-\lambda) y_2.
\end{align*}  
It is therefore enough to maximize
\begin{equation}\label{MAX2}
\beta_2\left[\lambda y_1^p + (1-\lambda) y_2^p\right] -\lambda I(y_1) - (1-\lambda)I(y_2)
\end{equation}
subject to the constraint
\begin{equation}\label{constraints2e}
\lambda y_1 + (1-\lambda)y_2 = e.
\end{equation}
If $\lambda \in \{0,1\}$ or 
$e \in \{0,1\}$, then~\eqref{constraints2e} 
shows that $g(x) \equiv e$. So assume that 
$e,\lambda \in (0,1)$. It is easy to see 
in this case that $0 < y_1 \le y_2 <1$. 
Moreover if $y_1 = y_2$ then again $g(x) \equiv e$, 
so assume $y_1 < y_2$. 
Introducing the Lagrange multiplier $\beta_1$, 
we find that
\begin{align*}
&-\lambda I'(y_1) + \lambda\beta_1  + \lambda p\beta_2 y_1^{p-1} = 0,\\
&-(1-\lambda) I'(y_2) + (1-\lambda)\beta_1  + (1-\lambda) p\beta_2 y_2^{p-1} = 0,\\
&-[I(y_1) - I(y_2)] + \beta_1(y_1-y_2) + \beta_2(y_1^p-y_2^p) = 0,
\end{align*}
and so
\begin{equation}\label{av}
 \ell'(y_1) = \ell'(y_2)=0 , \quad \ell(y_1) = \ell(y_2).
\end{equation}
If $\beta_2 \le \beta_2^c$, then there are no solutions to~\eqref{av}, 
while if $\beta_2 > \beta_2^c$, solutions 
occur precisely when $y_1 = x_1 < x_2 = y_2$ are the global 
maximizers of $\ell$ at the point $(q^{-1}(\beta_2),\beta_2)$ along $q$. 
In the latter case,~\eqref{constraints2e} implies
\begin{equation*}
x_1 \le e \le x_2.
\end{equation*}
If $e = x_1$ or $e = x_2$, 
then $g(x) \equiv e$. Since for 
$(e,\beta_2) \in U_e^c$ we have (by definition) 
$e \notin (x_1,x_2)$, we have established that 
~\eqref{psiforme} is valid in $U_e^c$.

Next, for $(e,\beta_2) \in U_e$ define
\begin{equation*}
 \lambda = \frac{x_2-e}{x_2-x_1}
\end{equation*}
and
\begin{equation*}
H(e) = \beta_{2}e^{p}-I(e) - \left(\beta_{2}\left[\lambda x_{1}^{p}
+(1-\lambda)x_{2}^{p}\right]
-\left[\lambda I(x_{1})
+(1-\lambda) I(x_{2})\right]\right).
\end{equation*}
To establish that~\eqref{psiforme} is valid in $U_e$, 
it suffices to show that $H(e) < 0$ for all $e \in (x_1,x_2)$. 
It is easy to check that $H(x_1) = H(x_2) = 0$ and 
\begin{align*}
 H'(e) &= p\beta_2 e^{p-1}-I'(e)- \frac{1}{x_2-x_1}\left(\beta_2(x_2^p-x_1^p)-
(I(x_2)-I(x_1))\right)\\
&= \ell'(e) - \beta_1 - \frac{1}{x_2-x_1}\left(\ell(x_2)-\ell(x_1)+ \beta_1(x_2-x_1)\right)\\
&= \ell'(e).
\end{align*}
Thus, $H'(x_1) = H'(x_2) = 0$. Finally, 
\begin{equation*}
H''(e) = p(p-1)\beta_2 e^{p-2} - I''(e) = \ell''(e).
\end{equation*}
From the proof of Proposition 11 in~\cite{AristoffZhu}, 
we know $\ell''(x_1) < 0$, $\ell''(x_2)<0$,
and moreover there exists $x_1 < u_1 < u_2 < x_2$ such that 
$\ell''(e) < 0$ for $e \in (x_1,u_1)\cup(u_2,x_2)$ 
while $\ell''(e) > 0$ for $e \in (u_1,u_2)$. The 
result follows.

We note that the star density can 
be computed easily as 
\begin{equation*}
s(e,\beta_{2}):=\frac{\partial}{\partial\beta_{2}}\psi(e,\beta_{2})=
\begin{cases}
e^{p}, &(e,\beta_{2})\in U_e^{c},
\\
\frac{e-x_{2}}{x_{1}-x_{2}}x_{1}^{p}
+\frac{x_{1}-e}{x_{1}-x_{2}}x_{2}^{p}, &(e,\beta_{2})\in U_e.
\end{cases}
\end{equation*}
It is clear that the star density $s(e,\beta_2)$ is continuous everywhere.

{\em Regularity of $\psi(e,\beta_2)$}. 
From the formula it is easy to see that $\psi(e,\beta_2)$ 
is analytic away from $\partial U_e$. 
Write $\psi = \psi(e,\beta_2)$ and fix $\beta_2 > \beta_2^c$. 
In the interior of $U_e^{c}$, 
\begin{equation*}
\frac{\partial\psi}{\partial e}
=p\beta_{2}e^{p-1}-I'(e),
\end{equation*}
while the interior of $U_e$,
\begin{align*}
\frac{\partial\psi}{\partial e}
&=\frac{\beta_{2}x_{1}^{p}}{x_{1}-x_{2}}-\frac{\beta_{2}x_{2}^{p}}{x_{1}-x_{2}}
-\frac{I(x_{1})}{x_{1}-x_{2}}+\frac{I(x_{2})}{x_{1}-x_{2}} \\
&= \frac{\ell(x_1)-\ell(x_2)}{x_1-x_2} - \beta_1 = -\beta_1.
\end{align*}
Observe that
\begin{equation*}
\lim_{(e,\beta_{2})\in U_e^{c},\,e\to x_{1}}\frac{\partial\psi}{\partial e}
=p\beta_{2}x_{1}^{p-1}-I'(x_{1}) = \ell'(x_1)-\beta_1 = -\beta_1.
\end{equation*}
So $\partial \psi/\partial e$ is continuous 
across $\partial U_e$ when $e < e^c$. Note that
\begin{equation}\label{zzero}
\frac{\partial^{j}\psi}{\partial e^{j}}=0,\qquad
(e,\beta_2)\in U_e ,\qquad j \ge 2.
\end{equation}
Proposition~11 
of~\cite{AristoffZhu} gives
\begin{equation*}
\lim_{(e,\beta_2)\in U_e^c,\, e \to x_1}
\partial^2 \psi(e,\beta_2)/\partial e^2 = \ell''(x_1) < 0.
\end{equation*}
Thus, $\partial^2\psi(e,\beta_2)/\partial e^2$ 
has a jump discontinuity across 
$\partial U_e$ when $e < e^c$. Analogous 
statements hold true for $e > e^c$.
Now consider the situation at 
the point $(e^c,\beta_2^c)$. Proposition~11 
of~\cite{AristoffZhu} gives
\begin{equation*}
\lim_{(e,\beta_2) \in U_e^c, \, e \to e^c}\frac{\partial^{4}\psi}{\partial e^{4}}
=\ell^{(4)}(e^c)<0.
\end{equation*}
From~\eqref{zzero} we see that $\partial^4 \phi/\partial e^4$ 
is discontinuous at $(e^c,\beta_2^c)$.

{\em Proof of part (ii).} 
By Theorem~\ref{varpsicanon}, 
\begin{equation}\label{functional3}
\psi(\beta_1,s) = \sup_{{\mathcal G}_{e,\cdot}} \left[\beta_1\int_0^1 g(x)\,dx - \int_0^1 I(g(x))\,dx\right]
\end{equation}
where we recall ${\mathcal G}_{\cdot,s}$ 
is the set of measurable functions $g:[0,1]\to [0,1]$ satisfying
\begin{equation}\label{constraint_s}
\int_0^1 g(x)^p\,dx = s.
\end{equation}
Arguments analogous to those in the proof of part {\em (i)} 
show that $\psi(\beta_1,s)$ has the formula~\eqref{psiforms}.

We note that the limiting edge density has the formula
\begin{equation*}
e(\beta_{1},s):=\frac{\partial}{\partial\beta_{1}}\psi(\beta_{1},s)=
\begin{cases}
s^{\frac{1}{p}} &(\beta_{1},s)\in U_s^{c},
\\
\frac{s-x_{2}^{p}}{x_{1}^{p}-x_{2}^{p}}x_{1}
+\frac{x_{1}^{p}-s}{x_{1}^{p}-x_{2}^{p}}x_{2} &(\beta_{1},s)\in U_s.
\end{cases}
\end{equation*}
It is easy to see that $e(\beta_{1},s)$ is continuous everywhere.

{\em Regularity of $\psi(\beta_1,s)$}. 
From the formula it is easy to see that $\psi(\beta_1,s)$ 
is analytic away from $\partial U_s$.
Write $\psi = \psi(\beta_1,s)$ and fix $\beta_1 < \beta_1^c$. 
In the interior of $U_s^{c}$, 
\begin{equation*}
\frac{\partial\psi}{\partial s}
=\frac{1}{p}\beta_{1}s^{\frac{1}{p}-1}-\frac{1}{p}s^{\frac{1}{p}-1}I'(s^{\frac{1}{p}}),
\end{equation*}
and in the interior of $U_s$,
\begin{align*}
\frac{\partial\psi}{\partial s}
&=\frac{\beta_{1}x_{1}}{x_{1}^{p}-x_{2}^{p}}-\frac{\beta_{1}x_{2}}{x_{1}^{p}-x_{2}^{p}}
-\frac{I(x_{1})}{x_{1}^{p}-x_{2}^{p}}+\frac{I(x_{2})}{x_{1}^{p}-x_{2}^{p}}
 \\
 &= \frac{\ell(x_1)-\ell(x_2)}{x_1^p-x_2^p} - \beta_2 = -\beta_2.
\end{align*}
Note that
\begin{align*}
\lim_{(\beta_{1},s)\in U_s^{c},\,s\to x_{1}^{p}}\frac{\partial\psi}{\partial s}
&=\frac{1}{p}\beta_{1}x_{1}^{1-p}-\frac{1}{p}x_{1}^{1-p}I'(x_{1})
\\
&= \frac{x_1^{1-p}}{p}\left(\ell'(x_1)-p\beta_2 x_1^{p-1}\right) 
= -\beta_2.
\end{align*}
So $\partial \psi/\partial s$ is continuous across 
$\partial U_s$ when $s < s^c$. Note that 
\begin{equation}\label{zzero2}
\frac{\partial^{j}\psi}{\partial s^{j}}=0,\qquad
(\beta_{1},s)\in U_s, \qquad j \ge 2.
\end{equation}
In the computations below, let $(\beta_{1},s)\in U_s^{c}$
and $t=s^{\frac{1}{p}}$. Note that
\begin{equation*}
\frac{\partial\psi}{\partial s}=[\beta_{1}-I'(t)]\frac{\partial t}{\partial s}
=[\ell'(t)-\beta_{2}pt^{p-1}]\frac{\partial t}{\partial s},
\end{equation*}
and
\begin{align}\begin{split}\label{AAA}
\frac{\partial^{2}\psi}{\partial s^{2}}
&=[\ell'(t)-\beta_{2}pt^{p-1}]\frac{\partial^{2}t}{\partial s^{2}}
+\left[\ell''(t)-\beta_{2}p(p-1)t^{p-2}\right]
\left(\frac{\partial t}{\partial s}\right)^{2}
\\
&=\ell'(t)\frac{\partial^{2}t}{\partial s^{2}}
+\ell''(t)\left(\frac{\partial t}{\partial s}\right)^{2}.\end{split}
\end{align}
Using~\eqref{AAA} and Proposition~11 of~\cite{AristoffZhu}, 
we get 
\begin{equation*}
\lim_{(\beta_{1},s)\in U_s^{c},\,s\to x_{1}^{p}}\frac{\partial^2\psi}{\partial s^2}
= \ell''(x_1)\left(\frac{1}{p}x_1^{1-p}\right)^2 < 0.
\end{equation*}
Comparing with~\eqref{zzero2}, we see that 
$\partial^2\psi/\partial s^2$ has a 
jump discontinuity across $\partial U_s$ 
for $s < s^c$. Analogous results hold for  
$s > s^c$. 
Now consider the situation at 
the point $(\beta_1^c,s^c)$. From~\eqref{AAA},
\begin{equation*}
\frac{\partial^{3}\psi}{\partial s^{3}}
=\ell'(t)\frac{\partial^{3}t}{\partial s^{3}}
+\ell'''(t)\left(\frac{\partial t}{\partial s}\right)^{3}
+3\ell''(t)\frac{\partial t}{\partial s}\frac{\partial^{2}t}{\partial s^{2}},
\end{equation*}
and
\begin{align*}
\frac{\partial^{4}\psi}{\partial s^{4}}
&=4\ell''(t)\frac{\partial t}{\partial s}\frac{\partial^{3}t}{\partial s^{3}}
+\ell'(t)\frac{\partial^{4}t}{\partial s^{4}}
+\ell^{(4)}(t)\left(\frac{\partial t}{\partial s}\right)^{4}
\\
&\qquad\qquad\qquad
+6\ell'''(t)\left(\frac{\partial t}{\partial s}\right)^{2}\frac{\partial^{2}t}{\partial s^{2}}
+3\ell''(t)\left(\frac{\partial^{2}t}{\partial s^{2}}\right)^{2}.
\end{align*}
As $s\to s^c$, $t\to e^c$
and since $\ell'(e^c)=\ell''(e^c)=\ell'''(e^c)=0$, $\ell^{(4)}(e^c) < 0$,
we have
\begin{equation*}
\lim_{s\to s^c}\frac{\partial^{2}\psi}{\partial s^{2}}
=\lim_{s\to s^c}\frac{\partial^{3}\psi}{\partial s^{3}}=0,
\end{equation*}
while Proposition 11 of~\cite{AristoffZhu} gives
\begin{equation*}
\lim_{s\to s^c}\frac{\partial^{4}\psi}{\partial s^{4}}
=\lim_{s\to s^c}\ell^{(4)}(t)\left(\frac{\partial t}{\partial s}\right)^{4}
=\ell^{(4)}(e^c)\left(\frac{(e^c)^{1-p}}{p}\right)^{4} < 0.
\end{equation*}

\end{proof}

The next result concerns the curve $\beta_2 = q(\beta_1)$ and the shapes of
$U_e$ and $U_s$. 

\begin{proposition}\label{prop}
(i) The curve $\beta_2 = q(\beta_1)$ is analytic 
in $\beta_1 < \beta_1^c$. 

(ii) For any $\beta_{1}<\beta_{1}^{c}$, $\frac{\partial x_{1}}{\partial\beta_{1}}>0$ and
$\frac{\partial x_{2}}{\partial\beta_{1}}<0$. Moreover,
\begin{align*}
&\lim_{\beta_{1}\rightarrow\beta_{1}^{c}}\frac{\partial x_{1}}{\partial\beta_{1}}=+\infty,
\qquad
\lim_{\beta_{1}\rightarrow-\infty}\frac{\partial x_{1}}{\partial\beta_{1}}=0,
\\
&\lim_{\beta_{1}\rightarrow\beta_{1}^{c}}\frac{\partial x_{2}}{\partial\beta_{1}}=-\infty,
\qquad
\lim_{\beta_{1}\rightarrow-\infty}\frac{\partial x_{2}}{\partial\beta_{1}}=0.
\end{align*}
For any $\beta_{2}>\beta_{2}^{c}$, $\frac{\partial x_{1}}{\partial\beta_{2}}<0$ and
$\frac{\partial x_{2}}{\partial\beta_{2}}>0$. Moreover,
\begin{align*}
&\lim_{\beta_{2}\rightarrow\beta_{2}^{c}}\frac{\partial x_{1}}{\partial\beta_{2}}=-\infty,
\qquad
\lim_{\beta_{2}\rightarrow+\infty}\frac{\partial x_{1}}{\partial\beta_{2}}=0,
\\
&\lim_{\beta_{2}\rightarrow\beta_{2}^{c}}\frac{\partial x_{2}}{\partial\beta_{2}}=+\infty,
\qquad
\lim_{\beta_{2}\rightarrow+\infty}\frac{\partial x_{2}}{\partial\beta_{2}}=0.
\end{align*}
\end{proposition}

\begin{proof}
First we show that $q$ is analytic. 
There is an open V-shaped set containing the 
phase transition curve $\beta_2 = q(\beta_1)$ 
except the critical point $(\beta_1^c,\beta_2^c)$, 
inside which $\ell$ has exactly two 
local maximizers, $y_1 < y_2$. (See~\cite{Radin} 
and~\cite{AristoffZhu}.) 
It can be seen from 
the proof of Proposition 11 of~\cite{AristoffZhu} 
that $\ell''(y_1)<0$ and $\ell''(y_2)<0$ inside 
the V-shaped region. The analytic 
implicit function theorem~\cite{Krantz} then shows that 
$y_1$ and $y_2$ are analytic functions of $\beta_1$ 
and $\beta_2$ inside this region. 
Note that $q$ is defined implicitly by the 
equation 
\begin{equation*}
\beta_1 y_1 + \beta_2 y_1^p - I(y_1)- (\beta_1 y_2 + \beta_2 y_2^p - I(y_2)) = 0.
\end{equation*}
Differentiating the left hand 
side of this equation w.r.t. $\beta_2$ 
gives 
\begin{align*}
&\beta_1 \frac{\partial y_1}{\partial \beta_2} + y_1^p + 
\left(p\beta_2y_1^{p-1}-I'(y_1)\right) \frac{\partial y_1}{\partial \beta_2}\\
&\quad\quad- \left[\beta_1 \frac{\partial y_2}{\partial \beta_2} + y_2^p + 
\left(p\beta_2y_2^{p-1}-I'(y_2)\right) \frac{\partial y_2}{\partial \beta_2}\right]\\
&= y_1^p - y_2^p < 0.
\end{align*}
Another application of the analytic 
implicit function theorem implies $\beta_2 = q(\beta_1)$ 
is analytic for $\beta_1 < \beta^c$.

Now we turn to the statements involving $x_1$ and $x_2$. 
Along $\beta_2 = q(\beta_1)$, we have 
\begin{equation*}
\beta_{1}+pq(\beta_{1})x_{1}^{p-1}-\log\left(\frac{x_{1}}{1-x_{1}}\right)=0.
\end{equation*}
Differentiating with respect to $\beta_{1}$, we get
\begin{equation*}
1+pq'(\beta_{1})x_{1}^{p-1}+\left[p(p-1)q(\beta_{1})x_{1}^{p-2}
-\frac{1}{x_{1}(1-x_{1})}\right]\frac{\partial x_{1}}{\partial\beta_{1}}=0.
\end{equation*}
Therefore,
\begin{equation*}
\frac{\partial x_{1}}{\partial\beta_{1}}
=\frac{1+pq'(\beta_{1})x_{1}^{p-1}}{\frac{1}{x_{1}(1-x_{1})}-p(p-1)q(\beta_{1})x_{1}^{p-2}}
=\frac{1-\frac{x_{1}-x_{2}}{x_{1}^{p}-x_{2}^{p}}px_{1}^{p-1}}{-\ell''(x_{1})}>0.
\end{equation*}
Since $\lim_{\beta_{1}\rightarrow\beta_{1}^{c}}[1+pq'(\beta_{1})x_{1}^{p-1}]
=\lim_{\beta_{1}\rightarrow\beta_{1}^{c}}[-\ell''(x_{1})]=0$, 
by L'H\^{o}pital's rule, 
\begin{equation*}
\lim_{\beta_{1}\rightarrow\beta_{1}^{c}}\frac{\partial x_{1}}{\partial\beta_{1}}
=\lim_{\beta_{1}\rightarrow\beta_{1}^{c}}\frac{pq''(\beta_{1})x_{1}^{p-1}
+p(p-1)q'(\beta_{1})x_{1}^{p-2}\frac{\partial x_{1}}{\partial\beta_{1}}}
{-p(p-1)q'(\beta_{1})x_{1}^{p-2}-\ell'''(x_{1})\frac{\partial x_{1}}{\partial\beta_{1}}},
\end{equation*}
which implies that
\begin{equation*}
\lim_{\beta_{1}\rightarrow\beta_{1}^{c}}\frac{\partial x_{1}}{\partial\beta_{1}}=+\infty.
\end{equation*}
Since $x_{1}\rightarrow 0$ as $\beta_{1}\rightarrow\beta_{1}^{c}$, it is easy to see that
\begin{equation*}
\lim_{\beta_{1}\rightarrow-\infty}\frac{\partial x_{1}}{\partial\beta_{1}}=0.
\end{equation*}
The results for $x_{2}$ can be proved using the similar methods. 
Finally, notice that
\begin{equation*}
\frac{\partial x_{i}}{\partial\beta_{2}}
=\frac{\partial x_{i}}{\partial\beta_{1}}\frac{\partial q^{-1}(\beta_{2})}{\partial\beta_{2}}
=\frac{\partial x_{i}}{\partial\beta_{1}}\frac{1}{q'(\beta_{1})}, \quad i=1,2.
\end{equation*}
Therefore the results involving $\beta_{2}$ also hold.
\end{proof}

\section*{Acknowledgements}

David Aristoff was supported in part by 
AFOSR Award FA9550-12-1-0187.
Lingjiong Zhu is grateful to his colleague Tzu-Wei Yang for helpful discussions. 
The authors are grateful for useful comments from the 
anonymous referees and associate editor.

\end{document}